\documentclass[12pt,reqno]{amsart}
\usepackage{amsfonts}
\usepackage{amsmath}
\usepackage{amssymb}
\usepackage{amscd}
\usepackage{graphicx}
\usepackage{hyperref}

\newtheorem{theorem}{Theorem}[section]
\newtheorem{lemma}{Lemma}[section]

\newtheorem{definition}{Definition}[section]
\newtheorem{example}{Example}[section]

\numberwithin{equation}{section}

\begin{document}
\title[Finite Blaschke Products]{Finite Blaschke Products and the Golden
Ratio}
\author{N\.{I}HAL YILMAZ \"{O}ZG\"{U}R}
\address{Bal\i kesir University\\
Department of Mathematics\\
10145 Bal\i kesir, TURKEY}
\email{nihal@balikesir.edu.tr}
\author{S\"{U}MEYRA U\c{C}AR}
\address{Bal\i kesir University\\
Department of Mathematics\\
10145 Bal\i kesir, TURKEY}
\email{sumeyraucar@balikesir.edu.tr}
\date{}

\begin{abstract}
It is known that the golden ratio $\alpha =\frac{1+\sqrt{5}}{2}$ has many
applications in geometry. In this paper we consider some geometric
properties of finite Blaschke products related to the golden ratio.
\end{abstract}

\maketitle

\section{\textbf{Introduction}}

\label{sec:1}

The golden ratio $\alpha =\frac{1+\sqrt{5}}{2}$ is the positive root of the
quadratic equation $x^{2}-x-1=0$. So we have
\begin{equation}
\alpha ^{2}=\alpha +1.  \label{eqn1}
\end{equation}

The golden ratio appears in modern research in many fields. For example, in
\cite{tutte} the golden ratio is used in graphs, in \cite{hopkins} it is
proved that in any dimension all solutions between unity and the golden
ratio to the optimal spherical code problem for $N$ spheres are also
solutions to the corresponding DLP (the densed local packing problem)
problem.

In this paper we give a connection between geometric properties of Blaschke
products and the golden ratio.

The rational function
\begin{equation*}
B(z)=\beta \underset{i=1}{\overset{n}{\prod }}\frac{z-a_{i}}{1-\overline{%
a_{i}}z}
\end{equation*}%
is called a finite Blaschke product of degree $n$ for the unit disc where $%
\left\vert \beta \right\vert =1$ and $\left\vert a_{i}\right\vert <1$, $%
1\leq i\leq n$. We call the finite Blaschke products of the following form
as canonical:%
\begin{equation}
B(z)=z\underset{j=1}{\overset{n-1}{\prod }}\frac{z-a_{j}}{1-\overline{a_{j}}z%
},\left\vert a_{j}\right\vert <1\text{ for }1\leq j\leq n-1.  \label{eqn11}
\end{equation}%
Note that the canonical Blaschke products correspond to finite Blaschke
products vanishing at the origin.

It is well-known that every Blaschke product $B$ of degree $n$ with $B(0)=0,$
is associated with a unique Poncelet curve (for more details see \cite%
{Daepp2002}, \cite{Daepp} and \cite{gau}). From \cite{Daepp2002} we know
that the Poncelet curve associated with a Blaschke product of degree $3$ is
an ellipse.

Here we investigate the relationships between the zeros of these canonical
finite Blaschke products and the golden ratio for $n=2,3,4$. Also we give
some examples for the cases $n=5,10$.

\section{\textbf{Blaschke Products of Degree Two}}

\label{sec:2}

Let $AB\ $be a line segment and $C$ be a point on the line segment $AB\ $%
such that $AC\ $is the greater part of $AB.$ Recall that we say the point $C$
divides the line segment $AB$ in the golden ratio if $\frac{AC}{BC}=\alpha $
\cite{Koshy}.

In this section we consider a finite Blaschke product $B$ of degree two of
the form
\begin{equation}
B_{a}(z)=z\frac{z-a}{1-\overline{a}z},  \label{eqn22}
\end{equation}%
with $a\neq 0$, $\left\vert a\right\vert <1$. From \cite{Daepp2002}, we know
that there exist two distinct points $z_{1}$ and $z_{2}$ on $\partial
\mathbb{D}$ that $B_{a}(z)$ maps to $\lambda $, for any point $\lambda $ on
the unit circle $\partial \mathbb{D}$, and that the line joining $z_{1}$ and
$z_{2}$ passes through $a$, the nonzero zero of $B_{a}$. Conversely, let $L$
be any line through the point $a$, then for the points $z_{1}$ and $z_{2}$
at which $L$ intersects $\partial \mathbb{D}$ we have $%
B_{a}(z_{1})=B_{a}(z_{2})$.

Now we ask the following questions:

$1$) Does the point $a$ divide the line segment $\left[ z_{1},z_{2}\right] $
joining $z_{1}$ and $z_{2}$ in the golden ratio?

$2$) If it does, what is the number of these line segments?

The answers of these questions are given in the following theorem.

\begin{figure}[h]
\centering
\includegraphics[height=8cm, width=8cm]{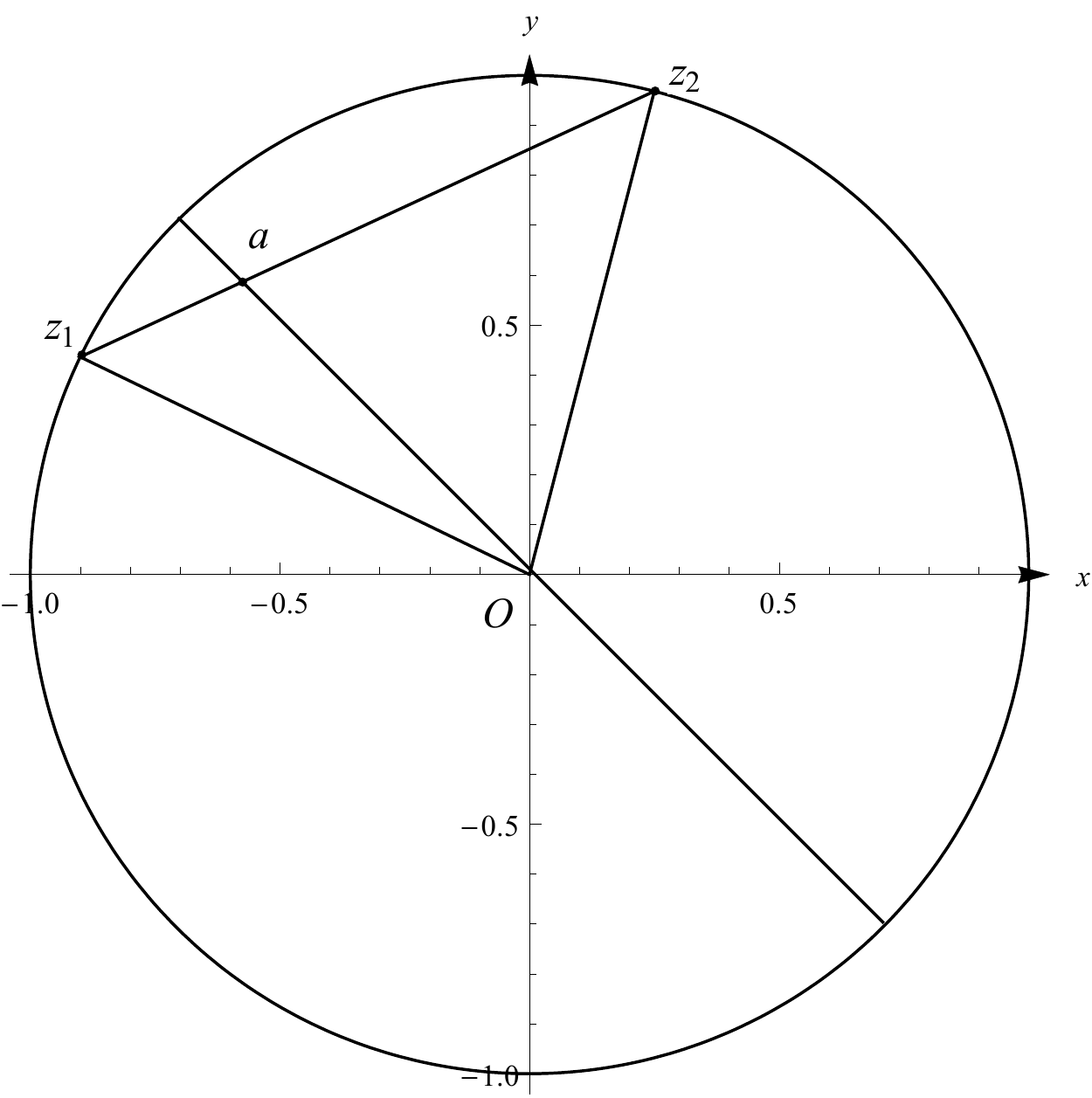}
\caption{{\protect\small {}}}
\label{fig:4}
\end{figure}

\begin{theorem}
\label{thm1} Let $B_{a}(z)=z\frac{z-a}{1-\overline{a}z}$ be a Blaschke
product with $a\neq 0$, $\left\vert a\right\vert <1$. There are infinitely
many values of $a$ such that there is a line segment with endpoints on the
unit circle divided by $a$ in the golden ratio. Furthermore the number of
such line segments is at most two for a fixed $a$.
\end{theorem}

\begin{proof}
Let $a$ be a fixed point such that $a\neq 0$, $\left\vert a\right\vert <1$
and consider the finite Blaschke product $B_{a}(z)=z\frac{z-a}{1-\overline{a}%
z}$. The ratio of the length of the longer part to length of the smaller
part of the segment $\left[ z_{1},z_{2}\right] $ divided by the point $a$
gives rise to a continuous function of the angle $\theta $ between the
segments $\left[ 0,a\right] $ and $\left[ z_{1},z_{2}\right] $. For $\theta
=0$ the ratio is $\frac{1+\left\vert a\right\vert }{1-\left\vert
a\right\vert }$ and for $\theta =\frac{\pi }{2}$ the ratio is $1$. Applying
the well known secant property of a circle to Figure \ref{fig:4}, it should
be
\begin{equation*}
(1-\left\vert a\right\vert )(1+\left\vert a\right\vert )=l\alpha .l
\end{equation*}%
where $l$ is the length of the segment $\left[ z_{1},a\right] $ and $l\alpha
$ is the length of the segment $\left[ a,z_{2}\right] $ \cite{coxeter}. Then
we get
\begin{equation}
l=\sqrt{\frac{1-\left\vert a\right\vert ^{2}}{\alpha }}  \label{eqn4}
\end{equation}%
Since nonlinear three distinct points determine a triangle, if the points $%
0,z_{1},z_{2}$ form a triangle it should be
\begin{equation}
0<l+l\alpha <2.  \label{eqn5}
\end{equation}%
If we substitute the equation $($\ref{eqn4}$)$ in $($\ref{eqn5}$),$ we get
\begin{equation*}
\sqrt{\frac{1-\left\vert a\right\vert ^{2}}{\alpha }}(\alpha +1)<2.
\end{equation*}%
Then we get
\begin{equation*}
\frac{1+\left\vert a\right\vert }{1-\left\vert a\right\vert }>\alpha .
\end{equation*}%
If $\alpha =\frac{1+\left\vert a\right\vert }{1-\left\vert a\right\vert },$
then the line passing through the points $z_{1},z_{2}$ and $a$ is the
diameter of the unit circle.

For this reason, as long as $\frac{1+\left\vert a\right\vert }{1-\left\vert
a\right\vert }\geq \alpha ,$ there is a segment divided by $a$ in the golden
ratio. Now we find the number of the segments divided by $a$ in the golden
ratio for a such $a$.

Let $a$ be chosen such that $\frac{1+\left\vert a\right\vert }{1-\left\vert
a\right\vert }\geq \alpha $ and $z_{1}$ be chosen such that the point $a$
divides the line segment $\left[ z_{1},z_{2}\right] $ in the golden ratio.
Then by definition we have%
\begin{equation}
\frac{\left\vert z_{2}-a\right\vert }{\left\vert z_{1}-a\right\vert }=\alpha
.  \label{eqn333}
\end{equation}%
Using the fact that $\left\vert z\right\vert =1$ for $z\in \partial \mathbb{D%
}$ we can write
\begin{equation*}
B(z)=\frac{z-a}{\overline{z}-\overline{a}}\text{, }z\in \partial \mathbb{D}.
\end{equation*}%
Also we know that $B(z_{1})=B(z_{2})$ and so we obtain
\begin{equation}
\frac{(z_{1}-a)}{(\overline{z}_{1}-\overline{a})}=\frac{z_{2}-a}{\overline{z}%
_{2}-\overline{a}}  \label{eqn31}
\end{equation}%
From the equation (\ref{eqn333}) we have
\begin{equation}
\frac{(z_{2}-a)(\overline{z}_{2}-\overline{a})}{(z_{1}-a)(\overline{z}_{1}-%
\overline{a})}=\alpha ^{2}  \label{eqn33}
\end{equation}%
and from the equation (\ref{eqn31}) we find
\begin{equation}
\overline{z}_{2}-\overline{a}=\frac{(z_{2}-a)(\overline{z}_{1}-\overline{a})%
}{(z_{1}-a)}  \label{eqn32}
\end{equation}%
After substitute (\ref{eqn32}) into (\ref{eqn33}) we get the equation%
\begin{equation}
-z_{2}^{2}+2az_{2}+\alpha a^{2}-2a\alpha ^{2}z_{1}+\alpha ^{2}z_{1}^{2}=0%
\text{.}  \label{eqn34}
\end{equation}%
Clearly the last equation (\ref{eqn34}) has at most two roots with respect
to $z_{2}$. Hence there are at most two line segments $\left[ z_{1},z_{2}%
\right] $ passing through the point $a$ and divided by $a$ in the golden
ratio. This fact can be also seen by some geometric arguments.
\end{proof}

\begin{figure}[t]
\centering
\includegraphics[height=8cm, width=8cm]{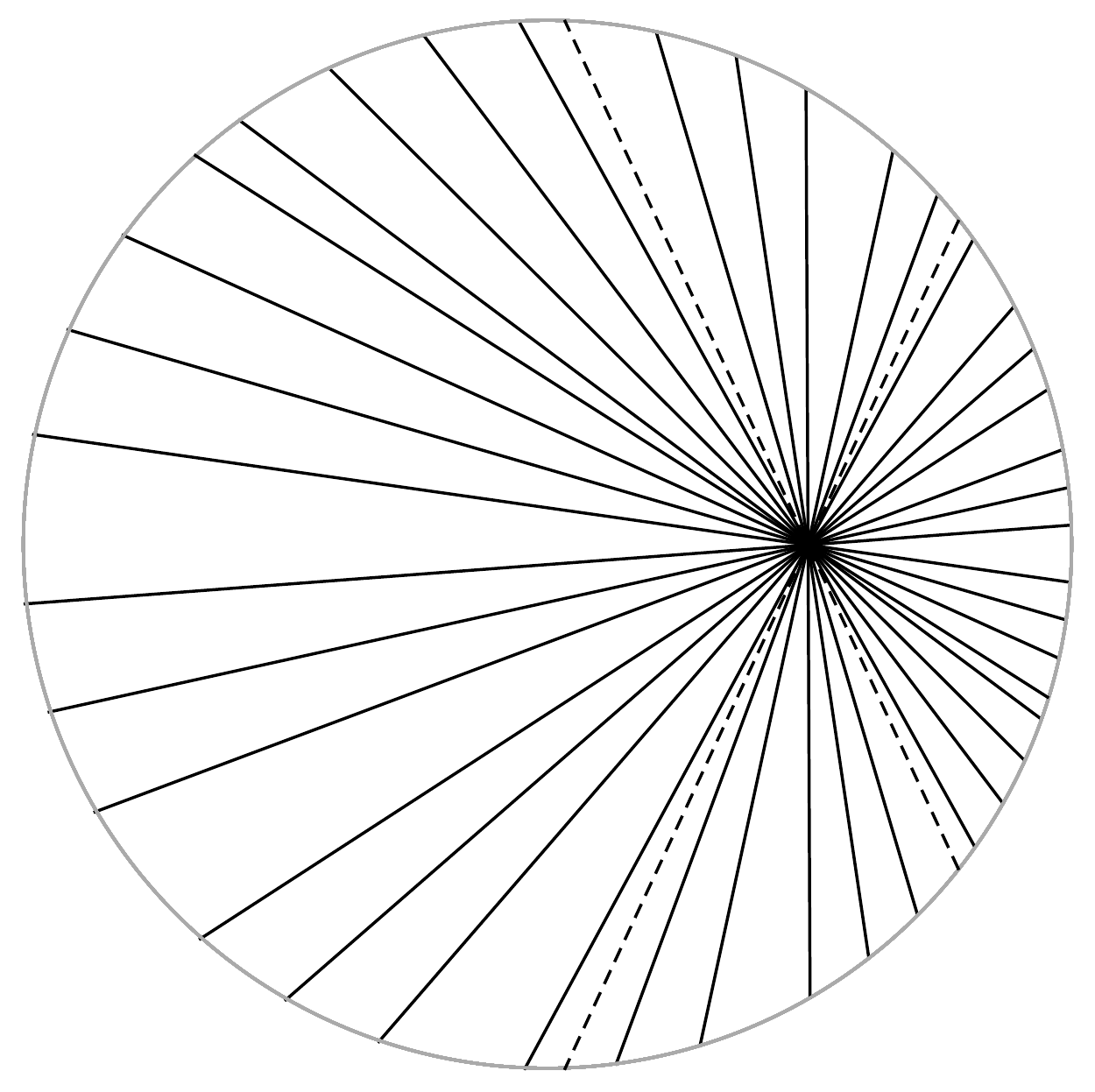}
\caption{{\protect\small {Blaschke product }}of degree $2$ with $a=\frac{1}{2%
}${\protect\small {.}}}
\label{fig:1}
\end{figure}

\begin{example}
Let us consider the Blaschke product $B(z)=z\frac{z-\frac{1}{2}}{1-\frac{1}{2%
}z}$. Let $z_{1}$ and $z_{2}$ be two distinct points satisfying $%
B(z_{1})=B(z_{2})$. If the point $a=\frac{1}{2}$ divides the line segment $%
\left[ z_{1},z_{2}\right] $ in the golden ratio, from the common solutions
of the equations $($\ref{eqn33}$)$ and $($\ref{eqn31}$)$ we obtain the
Figure \ref{fig:1}. There the dashed line segments show the line segments
which are divided by the point $a=\frac{1}{2}$ in the golden ratio.
\end{example}

\section{\textbf{Blaschke Products of Degree Three}}

\label{sec:3}

In this section, we consider a finite Blaschke product $B$ of degree three
of the form
\begin{equation*}
B(z)=z\frac{(z-a_{1})(z-a_{2})}{(1-\overline{a}_{1}z)(1-\overline{a}_{2}z)},
\end{equation*}%
with distinct zeros at the points $0$, $a_{1}$ and $a_{2}$. It is well-known
that for any specified point $\lambda $ of the unit circle $\partial \mathbb{%
D}$, there exist $3$ distinct points $z_{1}$, $z_{2}$ and $z_{3}$ of $%
\partial \mathbb{D}$ such that $B(z_{1})=B(z_{2})=B(z_{3})=\lambda $.

We know the following theorem for a Blaschke product of degree three.

\begin{theorem}
$($See \cite{Daepp2002} Theorem 1$)$ Let $B$ be a Blaschke product of degree
three with distinct zeros at the points $0$, $a_{1}$ and $a_{2}$. For $%
\lambda $ on the unit circle, let $z_{1}$, $z_{2}$ and $z_{3}$ denote the
points mapped to $\lambda $ under $B$. Then the lines joining $z_{j}$ and $%
z_{k}$ for $j\neq k$ are tangent to the ellipse $E$ with equation%
\begin{equation}
\left\vert z-a_{1}\right\vert +\left\vert z-a_{2}\right\vert =\left\vert 1-%
\overline{a_{1}}a_{2}\right\vert.  \label{eqn8}
\end{equation}%
Conversely, every point on $E$ is the point of tangency of a line segment
joining two distinct points $z_{1}$ and $z_{2}$ on the unit circle for which
$B(z_{1})=B(z_{2})$.
\end{theorem}

The ellipse $E$ in (\ref{eqn8}) is called a Blaschke $3$-ellipse associated
with the Blaschke product $B(z)$ of degree $3$. There are many studies on
the ellipse $E$ given in $($\ref{eqn8}$)$ (see \cite{Daepp}, \cite{Frantz},
\cite{fujimura}, \cite{Shubak}, \cite{NYO1} and \cite{NYO2} for more
details). For any $\lambda \in \partial \mathbb{D}$, we know that $E$
circumscribed in the triangle $\Delta (z_{1},z_{2},z_{3})$, where $%
z_{1},z_{2}$ and $z_{3}$ are the points mapped to $\lambda $ under $B$.

A golden triangle is an isosceles triangle such that the ratio of one its
lateral sides to the base is the golden ratio $\alpha =\frac{1+\sqrt{5}}{2}.$
A golden ellipse is an ellipse such that the ratio of the major axis to the
minor axis is the golden ratio $\frac{1+\sqrt{5}}{2}$ (see \cite{Koshy} for
more details).

We have the following questions:

$1)$ Are there any Blaschke $3$-ellipses which are circumscribed (at least)
one golden triangle?

$2)$ Can a Blaschke $3$-ellipse be a golden ellipse? If so, what is the
number of these ellipses?

We begin by the answering of the first question.

\begin{theorem}
\label{thm2} There are infinitely many golden triangles whose three vertices
lie on the unit circle.
\end{theorem}

\begin{proof}
Without loss of generality, let $x$ and $y$ be chosen so that $x,y>0$ and
such that the triangle with vertices at the points $1,-x+iy,-x-iy$ is
inscribed in the unit circle. We try to determine the values of $x$ and $y$
such that $x^{2}+y^{2}=1$. By the definition of a golden triangle it is
sufficient to show that there are values of $x$ and $y$ on the unit circle
such that%
\begin{equation}
2\alpha y=\sqrt{y^{2}+(x+1)^{2}}.  \label{eqn9}
\end{equation}%
Squaring both sides of (\ref{eqn9}) and using the fact that $x^{2}+y^{2}=1$,
we obtain $2y^{2}\alpha ^{2}=x+1.$ Then we have%
\begin{equation*}
2(1-x^{2})\alpha ^{2}-x-1=0
\end{equation*}%
and so%
\begin{equation*}
2x^{2}\alpha ^{2}+x+(1-2\alpha ^{2})=0.
\end{equation*}

Solving this quadric equation for $x$ and $y,$ we obtain $x=0,809017$ and $%
y= $ $0.587785$ where $y=\sqrt{1-x^{2}}.$ So we have one golden triangle
such that its vertices are on the unit circle. Then there are infinitely
many golden triangles with vertices on the unit circle by rotation.
\end{proof}

Now we can construct some examples using some results from \cite{Daepp} and
\cite{Shubak}. Recall that two sets $\{z_{1},z_{2},...,z_{n}\}$ and $%
\{w_{1},w_{2},...,w_{n}\}$ of points from $\partial \mathbb{D}$ are
interspersed if $0\leq \arg (z_{1})<\arg (w_{1})<...<\arg (z_{n})<\arg
(w_{n})<2\pi $ $($see \cite{Daepp} for more details$)$.

From \cite{Frantz}, we know that the ellipses inscribed in triangles with
vertices on the unit circle are precisely Blaschke $3$-ellipses.

\begin{figure}[h]
\centering
\includegraphics[height=8cm, width=8cm]{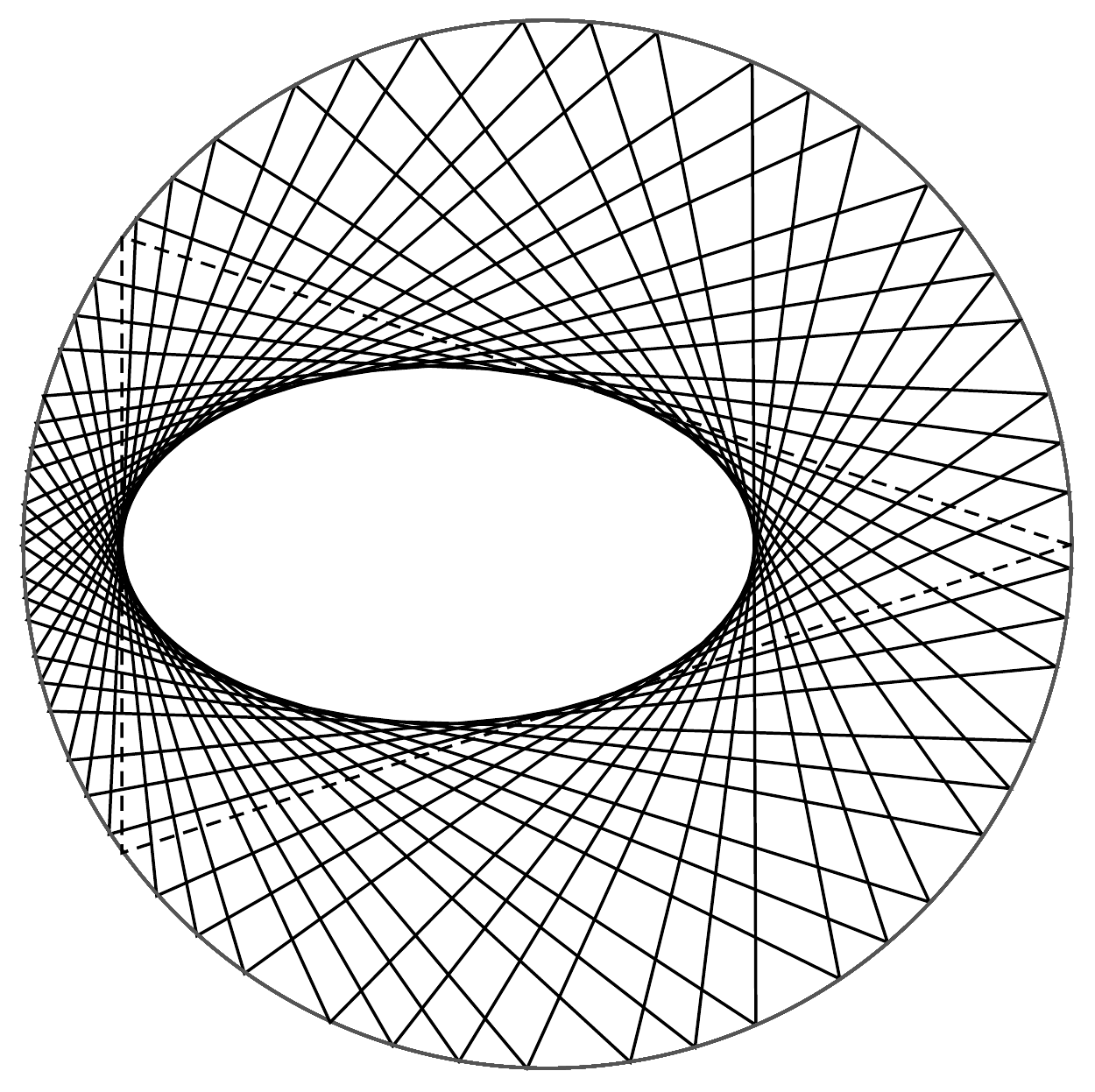}
\caption{{\protect\small Blaschke product $B$ of degree $3$ whose Poncelet
curve inscribed in (at least) one golden triangle. The dashed triangle is
the golden triangle. }}
\label{fig:2}
\end{figure}

\begin{example}
\label{ex:1} Let $\Delta (z_{1},z_{2},z_{3})$ be a golden triangle on the
unit circle. From Theorem $2.1$ in \cite{Shubak}, we know that the Steiner
ellipse $E$ inscribed in this golden triangle has foci $a_{1}$ and $a_{2}$
with the following equation $:$%
\begin{equation*}
a_{1}=\frac{1}{3}(z_{1}+z_{2}+z_{3})+\sqrt{(\frac{1}{3}%
(z_{1}+z_{2}+z_{3}))^{2}-\frac{1}{3}(z_{1}z_{2}+z_{1}z_{3}+z_{2}z_{3})}
\end{equation*}%
and
\begin{equation*}
a_{2}=\frac{1}{3}(z_{1}+z_{2}+z_{3})-\sqrt{(\frac{1}{3}%
(z_{1}+z_{2}+z_{3}))^{2}-\frac{1}{3}(z_{1}z_{2}+z_{1}z_{3}+z_{2}z_{3})}.
\end{equation*}%
Then this Steiner ellipse $E$ is the Poncelet curve of the Blaschke product $%
B(z)=z\frac{z-a_{1}}{1-\overline{a_{1}}z}\frac{z-a_{2}}{1-\overline{a_{2}}z}$%
.
\end{example}

\begin{figure}[h]
\centering
\includegraphics[height=8cm, width=8cm]{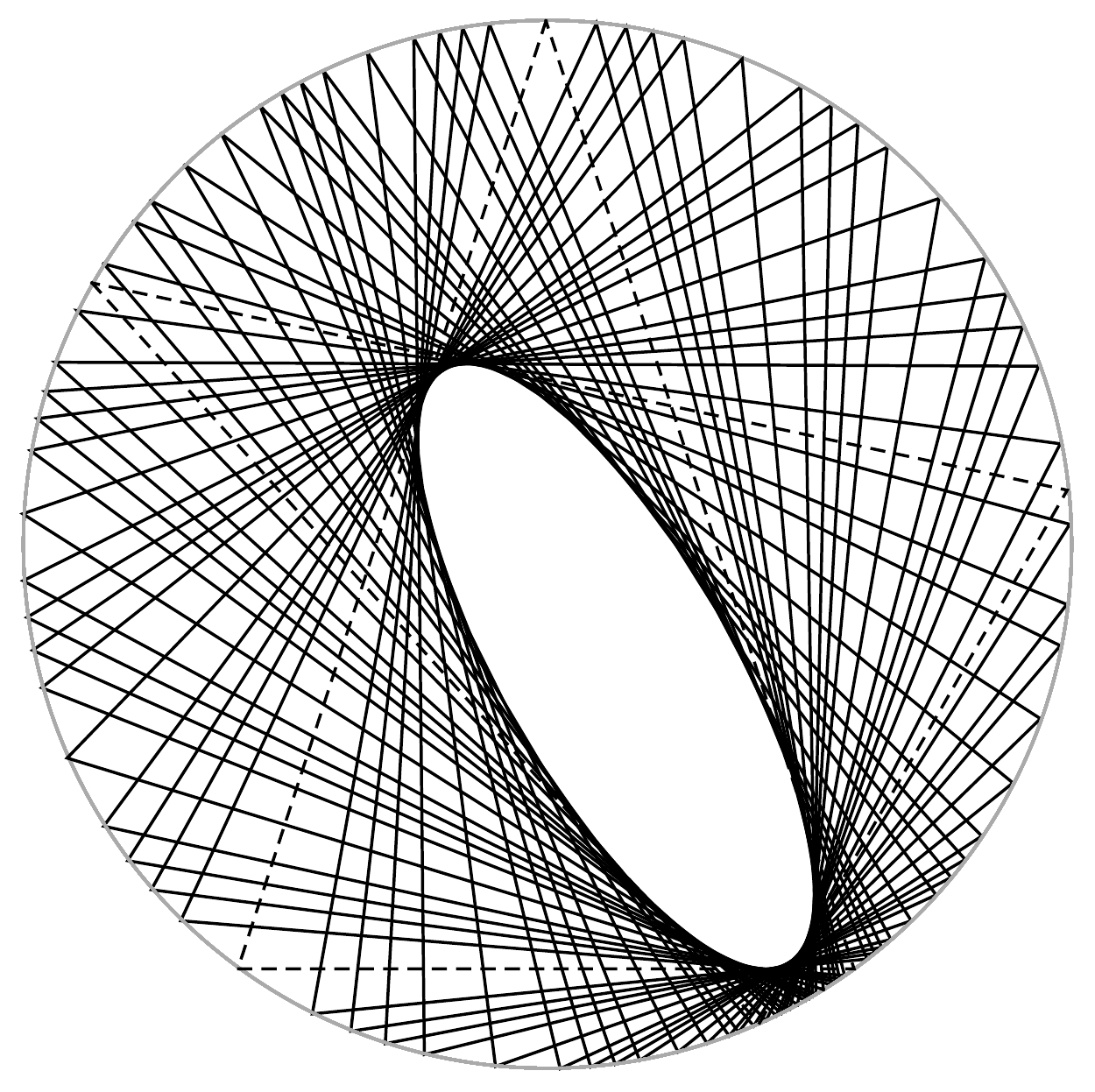}
\caption{{\protect\small Blaschke product $B$ of degree $3$ whose Poncelet
curve inscribed in (at least) two golden triangles. The dashed triangles are
the golden triangles. }}
\label{fig:22}
\end{figure}

\begin{example}
\label{ex:2} Let $z_{1},z_{2},z_{3}$ and $w_{1},w_{2},w_{3}$ be triples of
points which form the golden triangles $\Delta (z_{1},z_{2},z_{3})$ and $%
\Delta (w_{1},w_{2},w_{3})$ on the unit circle so that $\{z_{1},z_{2},z_{3}%
\} $ and $\{w_{1},w_{2},w_{3}\}$ are interspersed sets of the points. From
Corollary $10$ on page $97$ in \cite{Daepp}, we know that there exists a
Blaschke product $B$ of degree $3$ which maps $0$ to $0$ such that $%
B(z_{j})=B(z_{k})$ and $B(w_{j})=B(w_{k})$ for all $j$ and $k$ $(1\leq
j,k\leq 3).$ Since we can choose the triples $z_{1},z_{2},z_{3}$ and $%
w_{1},w_{2},w_{3}$ by infinitely many different ways then clearly there are
infinitely many Blaschke ellipses each of which has at least two golden
triangle circumscribing them and having the vertices on the unit circle.
\end{example}

We have seen examples of Blaschke products of degree three of which Poncelet
curves inscribed in at least one or two golden triangles.

Now we consider the answer of our second question.

\begin{theorem}
\label{thm7} There are infinitely many golden ellipses which are Blaschke
ellipses in the unit disc.
\end{theorem}

\begin{proof}
Let us take a golden ellipse with equation $\frac{x^{2}}{a^{2}}+\frac{y^{2}}{%
b^{2}}=1$ in the unit disc. Then by definition $\frac{a}{b}=\alpha $ and so $%
a=b\alpha $. Recall that we have the equation $a^{2}=b^{2}+c^{2}$ where the
point $c$ is the positive focus of the ellipse. So this ellipse has foci $-c$
and $c$. By the last equation and $\alpha ^{2}=\alpha +1$, we find $%
b^{2}\alpha =c^{2}$. Combining $a=b\alpha $ and $b^{2}\alpha =c^{2}$ we find
$a=\pm c\sqrt{\alpha }$. Now we consider the Blaschke product associated
with this ellipse. If this ellipse is a Blaschke ellipse, it must be $%
2a=1+c^{2}$ by the definition of a Blaschke ellipse. Hence we find $c^{2}\pm
2\sqrt{\alpha }c+1=0$. As these equations have only one positive root $c=%
\frac{1}{2}(-\sqrt{2(-1+\sqrt{5})}+\sqrt{2(1+\sqrt{5})})$, there is one
golden ellipse which is a Blaschke ellipse. Since every rotation of this
golden ellipse is again golden, clearly we have infinitely many golden
Blaschke ellipses in the unit disc.
\end{proof}

We give the following definition.

\begin{definition}
Let $B$ be a finite Blaschke product of degree $n$ of the canonical form. If
the Poncelet curve associated with $B$ is an ellipse and this ellipse is a
golden ellipse, then $B$ is called as a golden Blaschke product.
\end{definition}

\begin{example}
Let us consider the Blaschke product%
\begin{equation*}
B_{1}(z)=z\frac{(z-a_{1})(z-a_{2})}{(1-\overline{a}_{1}z)(1-\overline{a}%
_{2}z)},
\end{equation*}%
where
\begin{equation*}
a_{1}=\frac{1}{2}(-\sqrt{2(-1+\sqrt{5})}+\sqrt{2(1+\sqrt{5})})
\end{equation*}%
and $a_{2}=-a_{1}$. By the proof of Theorem \ref{thm7} we know that the
Blaschke $3$-ellipse $E$ associated with $B_{1}$ is a golden ellipse. So $%
B_{1}(z)$ is a golden Blaschke product. The image of this golden Blaschke
ellipse under the rotation transformation $f(z)=(\frac{1}{2}+i\frac{\sqrt{3}%
}{2})z$ is another golden Blaschke ellipse. Clearly we find the equation of $%
f(E)$ as%
\begin{equation*}
\left\vert z-(\frac{1}{2}+i\frac{\sqrt{3}}{2})a_{1}\right\vert +\left\vert
z-(\frac{1}{2}+i\frac{\sqrt{3}}{2})a_{2}\right\vert =\left\vert 1-\overline{%
a_{1}}a_{2}\right\vert .
\end{equation*}%
More precisely, this image ellipse $f(E)$ is the Poncelet curve of the
following Blaschke product $:$%
\begin{equation*}
B_{2}(z)=z\frac{(z-(\frac{1}{2}+i\frac{\sqrt{3}}{2})a_{1})(z-(\frac{1}{2}+i%
\frac{\sqrt{3}}{2})a_{2})}{(1-(\frac{1}{2}-i\frac{\sqrt{3}}{2})\overline{a}%
_{1}z)(1-(\frac{1}{2}-i\frac{\sqrt{3}}{2})\overline{a}_{2}z)}.
\end{equation*}
\end{example}

\section{\textbf{Blaschke Products of Degree Four}}

A golden rectangle is a rectangle such that the ratio of the length $x$ of
the longer side to the length $y$ of the shorter side is the golden ratio $%
\frac{1+\sqrt{5}}{2}$ (see \cite{Koshy} for more details).

We give the following theorem.

\begin{theorem}
There are infinitely many golden rectangles whose four vertices lie on the
unit circle.
\end{theorem}

\begin{proof}
Without loss of generality, let $x$ and $y$ be chosen so that $x,y>0$ and
such that the rectangle with vertices at the points $x+iy,x-iy,-x-iy,-x+iy$
is inscribed in the unit circle. We try to determine the values of $x$ and $%
y $ such that $x^{2}+y^{2}=1$. So it is sufficient to show that there are
values of $x$ and $y$ on the unit circle such that%
\begin{equation*}
2x=2\alpha y.
\end{equation*}%
We get $x=\alpha y$ and using the facts that $x^{2}+y^{2}=1$ and $\alpha
^{2}=\alpha +1$ we obtain $y^{2}(\alpha ^{2}+1)=1$ and hence%
\begin{equation*}
y=\frac{1}{\sqrt{\alpha +2}}=0.525731\text{ and }x=\frac{\alpha }{\sqrt{%
\alpha +2}}=0.850651\text{.}
\end{equation*}%
So we have one golden rectangle such that its vertices are on the unit
circle. Then there are infinitely many golden triangles with vertices on the
unit circle by rotation.
\end{proof}

\begin{example}
\label{ex:3} Let $z_{1},z_{2},z_{3},z_{4}$ and $w_{1},w_{2},w_{3},w_{4}$ be
eight points which form the golden rectangles $(z_{1},z_{2},z_{3},z_{4})$
and $(w_{1},w_{2},w_{3},w_{4})$ on the unit circle so that $%
\{z_{1},z_{2},z_{3},z_{4}\}$ and $\{w_{1},w_{2},w_{3},w_{4}\}$ are
interspersed sets of the points. From Corollary $10$ on page $97$ in \cite%
{Daepp}, we know that there exists a Blaschke product $B$ of degree $4$
which maps $0$ to $0$ such that $B(z_{j})=B(z_{k})$ and $B(w_{j})=B(w_{k})$
for all $j$ and $k$ $(1\leq j,k\leq 4).$ Then clearly there are infinitely
many Poncelet curves associated with a finite Blaschke product of degree $4$
each of which has at least two golden rectangle circumscribing them and
having the vertices on the unit circle.
\end{example}

\begin{figure}[h]
\centering
\includegraphics[height=8cm, width=8cm]{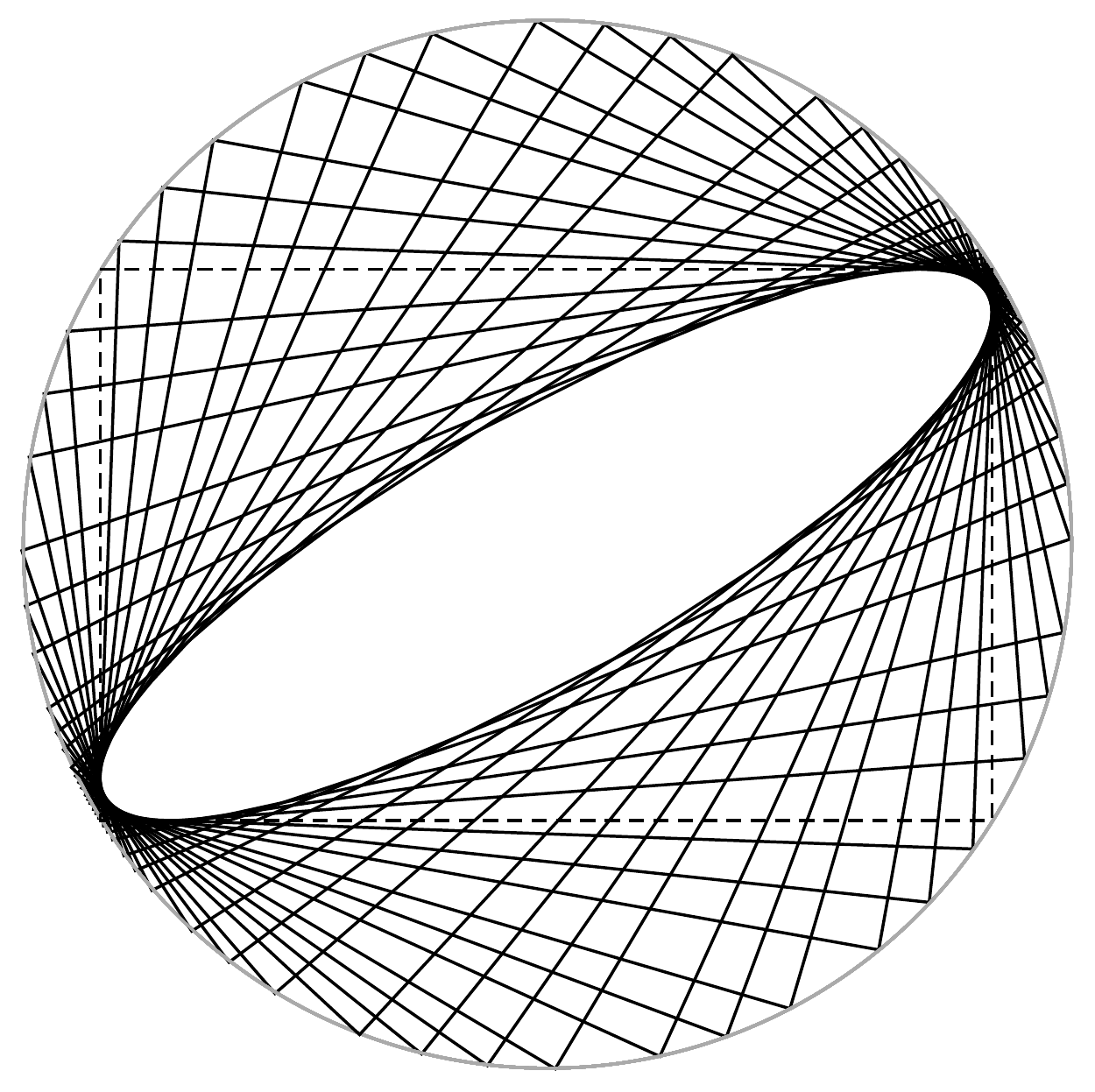}
\caption{{\protect\small {Blaschke product $B$ of degree $4$ whose Poncelet
curve is an ellipse inscribed in (at least) one golden rectangle. The dashed
rectangle is the golden rectangle.}}}
\end{figure}

Using the following lemmas, we construct examples of finite Blaschke
products of degree $4$ whose Poncelet curves are ellipses and each of them
have at least one golden rectangle.

\begin{lemma}
\label{lem1}$($See \cite{fujimura} Lemma $5)$ For any quadrilateral that is
inscribed in the unit circle, an ellipse is inscribed in it if and only if
the ellipse is associated with the composition of two Blaschke products of
degree $2.$
\end{lemma}

\begin{lemma}
\label{lem2}$($See \cite{fujimura} Lemma $6)$ For four mutually distinct
points $z_{1},...,z_{4}$ on the unit circle $(0\leq \arg z_{1}<\arg
z_{2}<\arg z_{3}<\arg z_{4}<2\pi ),$ there exists an ellipse that is
inscribed in the quadrilateral with vertices $z_{1},...,z_{4}.$ Moreover,
for each quadrilateral, inscribed ellipses form a real-valued one-parameter
family.
\end{lemma}

Now we give the following theorem.

\begin{theorem}
Let $Q$ be any golden rectangle inscribed in the unit circle. Then there is
at least one ellipse $E$ inscribed in $Q$ such that $E$ is a Poncelet curve
of a finite Blaschke product $B$ of degree $4$.
\end{theorem}

\begin{proof}
Let $Q$ be any golden rectangle with the vertices $z_{1},z_{2},z_{3},z_{4}$
on the unit circle. By Lemma \ref{lem2} there exists an ellipse $E$
inscribed in $Q$. We know that the two foci $a$ and $b$ of an ellipse
inscribed in any rectangle whose vertices are $z_{1},z_{2},z_{3},z_{4}$
satisfy the equations
\begin{equation*}
\begin{array}{c}
\left[ \left( \left( \left( -z_{2}+z_{1}\right) z_{3}-z_{1}z_{2}\right)
z_{4}+z_{1}z_{2}z_{3}\right) \overline{a}+z_{2}z_{4}-z_{1}z_{3}\right] a^{2}
\\
-[z_{1}z_{2}z_{3}z_{4}\left( z_{4}-z_{3}+z_{2}-z_{1}\right) \overline{a}%
^{2}-\left( z_{3}+z_{1}\right) \left( z_{4}+z_{2}\right) \left(
z_{2}z_{4}-z_{1}z_{3}\right) \overline{a} \\
+z_{2}z_{4}\left( z_{4}+z_{2}\right) -z_{1}z_{3}\left( z_{1}+z_{3}\right)
]a+z_{1}z_{2}z_{3}z_{4}\left( z_{2}z_{4}-z_{1}z_{3}\right) \overline{a}^{2}
\\
-[\left( z_{2}^{2}z_{3}+z_{1}z_{2}^{2}\right)
z_{4}^{2}-z_{1}^{2}z_{3}^{2}z_{4}-z_{1}^{2}z_{2}z_{3}^{2}]\overline{a} \\
+\left( z_{2}z_{4}-z_{1}z_{3}\right) \left( z_{2}z_{4}+z_{1}z_{3}\right) =0%
\end{array}%
\end{equation*}%
and
\begin{equation*}
\begin{array}{c}
\left( z_{4}-z_{3}+z_{2}-z_{1}\right) ab-\left( z_{2}z_{4}-z_{1}z_{3}\right)
\left( a+b\right) \\
+[\left( z_{2}-z_{1}\right) z_{3}+z_{1}z_{2}]z_{4}-z_{1}z_{2}z_{3}=0%
\end{array}%
\end{equation*}%
given in \cite{fujimura}. Then by the proof of Lemma \ref{lem1} $E$ has the
following equation
\begin{equation*}
E:\left\vert z-a\right\vert +\left\vert z-b\right\vert =\left\vert 1-%
\overline{a}b\right\vert \sqrt{\frac{\left\vert a\right\vert ^{2}+\left\vert
b\right\vert ^{2}-2}{\left\vert a\right\vert ^{2}\left\vert b\right\vert
^{2}-1}}
\end{equation*}%
and $E$ is the Poncelet curve of the finite Blaschke product $B$ of the
following form$:$%
\begin{equation*}
B(z)=z\frac{z-\beta }{1-\overline{\beta }z}\frac{z^{2}+(\overline{\beta }%
\alpha -\beta )z-\alpha }{1-(-\overline{\alpha }\beta +\overline{\beta })z-%
\overline{\alpha }z^{2}},
\end{equation*}%
where $\alpha =-ab$ and $\beta =\frac{a+b-ab(\overline{a}+\overline{b})}{%
1-\left\vert ab\right\vert ^{2}}.$
\end{proof}

\begin{figure}[h]
\centering
\includegraphics[height=8cm, width=8cm]{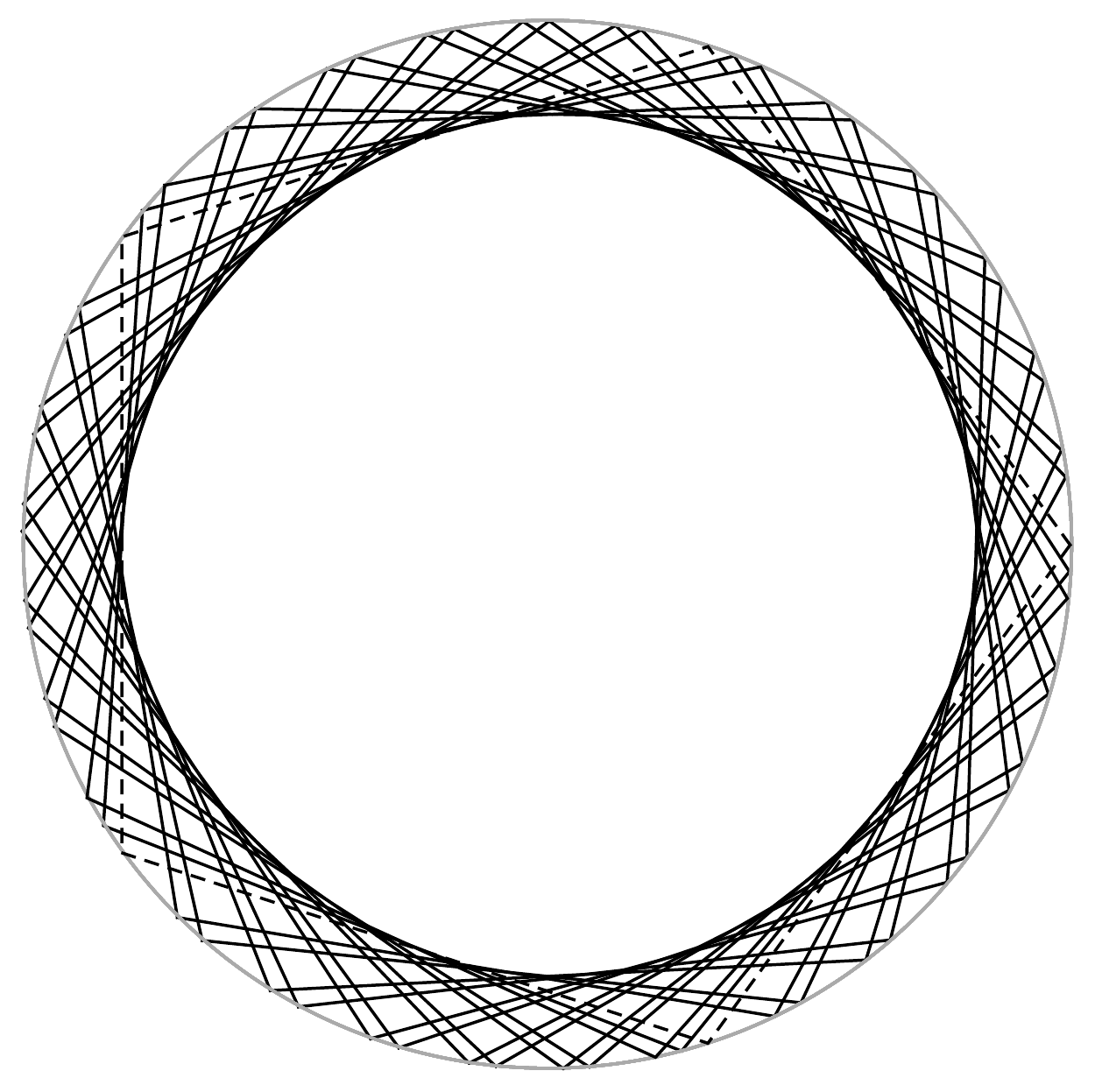}
\caption{{\protect\small Blaschke product $B$ of degree $5$ whose Poncelet
curve inscribed in (at least) one golden pentagon. The dashed pentagon is
the golden pentagon. }}
\label{fig:5}
\end{figure}

\section{Blaschke Products of Higher Degree}

\begin{figure}[h]
\centering
\includegraphics[height=8cm, width=8cm]{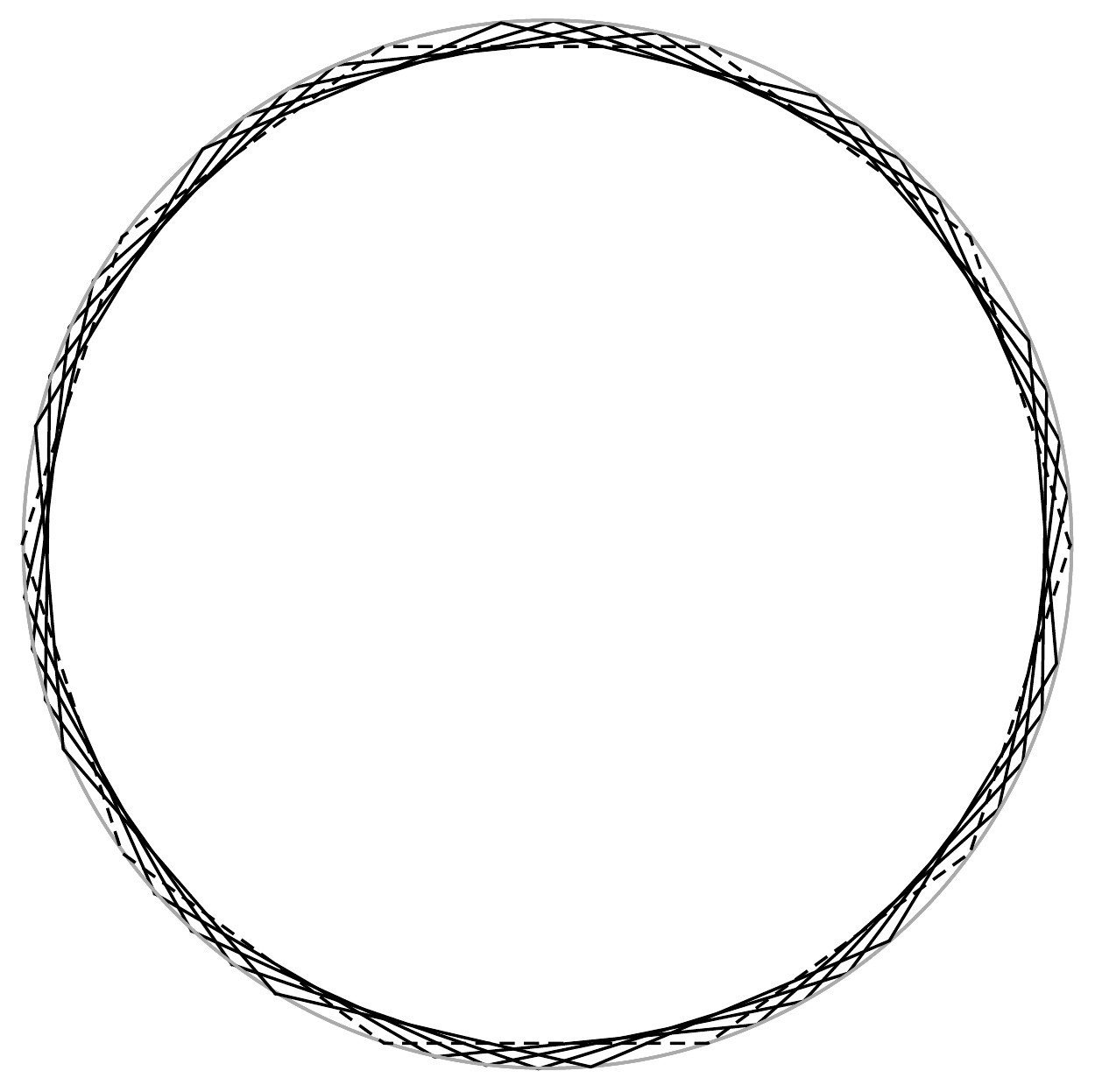}
\caption{{\protect\small Blaschke product $B$ of degree $10$ whose Poncelet
curve inscribed in (at least) one golden decagon. The dashed decagon is the
golden decagon. }}
\label{fig:6}
\end{figure}

We know that regular pentagon and regular decagon have the same properties
of the golden ratio among polygons (see \cite{Koshy} for more details). It
is not known the equation of the Poncelet curves of Blaschke products of
degree $5$ or $10$, so we cannot obtain similar theorems to the ones given
in the previous sections. In these two cases, by the similar arguments used
in the Example \ref{ex:2} and Example \ref{ex:3}, we can obtain finite
Blaschke products of degree $5$ and $10$ whose Poncelet curves circumscribed
by at least two regular pentagon and regular decagon, respectively.

\textbf{Acknowledgment.} The authors would like to thank Professor Pamela
Gorkin and Nathan Wagner for the discussion on the proof of Theorem \ref%
{thm2}.


\begin{thebibliography}{99}
\bibitem{coxeter} Coxeter H. S. M., \emph{Introduction to geometry}. Second
edition. John Wiley \& Sons, Inc., New York-London-Sydney, 1969.

\bibitem{Daepp2002} U. Daepp, P. Gorkin and R. Mortini, \emph{Ellipses and
finite Blaschke products}, Amer. Math. Monthly 109 (2002), no. 9, 785-795.

\bibitem{Daepp} U. Daepp, P. Gorkin and K. Voss, \emph{Poncelet's theorem,
Sendov's conjecture, and Blaschke products}, J. Math. Anal. Appl. 365\textbf{%
\ }(2010), no. 1, 93-102.

\bibitem{Frantz} M. Frantz, \emph{How conics govern M\"{o}bius
transformations}, Amer. Math. Monthly 111 (2004), no. 9, 779-790.

\bibitem{fujimura} M. Fujimura, \emph{Inscribed Ellipses and Blaschke
Products},Comput. Methods Funct. Theory 13 (2013), no. 4, 557-573.

\bibitem{gau} H. W. Gau, P. Y. Wu, \emph{Numerical Range and Poncelet
Property}, Taiwanese J. Math. 7 (2003), no 2, 173-193.

\bibitem{Shubak} P. Gorkin and E. Skubak, \emph{Polynomials, ellipses, and
matrices: two questions, one answer}, Amer. Math. Monthly 118 (2011), no. 6,
522-533.

\bibitem{hopkins} A. B. Hopkins, H. F. Stillinger and S. Torquato, \emph{%
Spherical codes, maximal local packing density, and the golden ratio, }J.
Math. Phys. 51 (2010), no. 4, 043302, 6 pp.

\bibitem{Koshy} T. Koshy, \emph{Fibonacci and Lucas numbers with applications%
}, Pure and Applied Mathematics (New York). Wiley-Interscience, New York,
2001.

\bibitem{NYO1} N. Y\i lmaz \"{O}zg\"{u}r, \emph{Finite Blaschke Products and
Circles that Pass Through the Origin}, Bull. Math. Anal. Appl. 3 (2011), no.
3, 64-72.

\bibitem{NYO2} N. Y\i lmaz \"{O}zg\"{u}r, \emph{Some Geometric Properties of
finite Blaschke Products}, Proceedings of the Conference RIGA 2011, (2011),
239-246.

\bibitem{tez} S. U\c{c}ar, \emph{Finite Blaschke Products and Some Geometric
Properties}, Ph.D. Thesis, Bal\i kesir University, 2015.

\bibitem{tutte} W. T. Tutte, \emph{On chromatic polynomials and the golden
ratio, }J. Emphasized Theory 9 (1970), 289-296.
\end{thebibliography}
\end{document}